\documentclass[12pt]{amsart}
\usepackage{amsmath,amssymb,amsthm,amsfonts,enumerate,array,color,lscape,fancyhdr,layout,pst-all}
\usepackage{xcolor}
\usepackage[english]{babel}
\usepackage{young}
\usepackage{float}
\usepackage[all]{xy}
\usepackage[normalem]{ulem}
\usepackage[active]{srcltx}
\usepackage{mathrsfs}
\usepackage{etoolbox}
\usepackage[margin=1in]{geometry}
\usepackage[colorlinks=true, breaklinks=true, allcolors=black]{hyperref}

\usepackage{xspace} 
\newlength\yStones
\newlength\xStones
\newlength\xxStones

\makeatletter
\def\Stones{\pst@object{Stones}}
\def\Stones@i#1{%
  \pst@killglue%
  \begingroup%
  \use@par%
  \setlength\xxStones{\xStones}%
  \expandafter\Stones@ii#1,,\@nil
  \endgroup
  \global\addtolength\xStones{0.6cm}%
  \global\addtolength\yStones{-7.5mm}}%
\def\Stones@ii#1,#2,#3\@nil{%
  \rput(\xxStones,\yStones){%
    \psframebox[framesep=0]{%
      \parbox[c][6mm][c]{11mm}{\makebox[11mm]{$#1$}}}}%
  \addtolength\xxStones{1.2cm}%
  \ifx\relax#2\relax\else\Stones@ii#2,#3\@nil\fi}
\makeatother

\def\Stone#1{\fbox{\makebox[12mm]{\strut#1}}\kern2pt}

\newcommand{\Z}{\mathbb{Z}}
\newcommand{\A}{\mathbb{A}}

\newcommand{\gln}{\mathfrak{gl}_n}
\newcommand{\cV}{\mathcal{V}}

\newcommand{\gl}{\mathfrak{gl}}

\newcommand{\Hom}{\mathrm{Hom} \,}
\newcommand{\End}{\mathrm{End}}
 
\newcommand{\Der}{\mathrm{Der}}

\newcommand{\fxpartial}[2]{\frac{\partial #1}{\partial #2}}
\newcommand{\xpartial}[1]{\frac{\partial}{\partial #1}}

\newcommand{\cD}{\mathcal{D}}
\newcommand{\cL}{\mathcal{L}}
\newcommand{\fm}{\mathfrak{m}}

\newcommand{\bk}{\Bbbk}
\newcommand{\Diff}{\mathrm{Diff}}
\newcommand{\GKdim}{\mathrm{GKdim}}

\newtheorem{theorem}{Theorem}[section]
\newtheorem{lemma}[theorem]{Lemma}
\newtheorem{corollary}[theorem]{Corollary}
\newtheorem{proposition}[theorem]{Proposition}

\theoremstyle{definition}
\newtheorem{example}[theorem]{Example}

\allowdisplaybreaks

\newtoggle{details}
\toggletrue{details}    
\iftoggle{details}{
\newcommand{\details}[1]{{\color{blue}\noindent\textbf{Details:}{#1}}}
                  }{
                  \newcommand{\details}[1]{}
                  }
                
\begin{document}
\begin{title}[Holonomic $A\cV$-modules for the affine space]{Holonomic $A\cV$-modules for the affine space}
\end{title}

\author[Y. Billig]{Yuly Billig}
\address{Carleton University \\ Ottawa \\ Canada}
\email{billig@math.carleton.ca}
\author[H. Rocha]{Henrique Rocha}
\email{henriquerocha@cunet.carleton.ca}


\begin{abstract}
We study the growth of representations of the Lie algebra of vector fields on the affine space that admit a compatible action of the polynomial algebra. We establish the Bernstein inequality for these representations, enabling us to focus on modules with minimal growth, known as holonomic modules. We show that simple holonomic modules are isomorphic to the tensor product of a holonomic module over the Weyl algebra and a finite-dimensional $\gl_n$-module. We also prove that holonomic modules have a finite length and that the representation map associated with a holonomic module is a differential operator. Finally, we present examples illustrating our results.
\end{abstract}


\subjclass[2020]{17B10, 17B66}


\maketitle


\tableofcontents    


\section*{Introduction}
In the realm of modern algebraic geometry and representation theory, the theory of $\cD$-modules has emerged as a powerful framework with far-reaching implications. In the algebraic context, a $\cD$-module is a module over the algebra of differential operators. It can be seen as a representation of the Lie algebra of vector fields together with an action of the algebra of functions. These two actions not only need to be compatible, satisfying the Leibniz rule, but the representation of the Lie algebra is required to be linear with respect to the action of the algebra of functions. 

Recently, a wider class of modules has been studied, the $A\cV$-modules. These can be seen as a generalization of $\cD$-modules since they have the same axioms except that the linearity condition of the representation of the Lie algebra of vector fields is not required. The development of a general theory of $AV$-modules on smooth algebraic varieties is presented in \cite{BF18,BFN19,BNZ21,BIN23,BR23, BI23, BB24}.

Holonomic modules form an especially important class of $\cD$-modules. For instance, one of their applications appeared in the Riemann-Hilbert correspondence, which provides a relation between certain types of holonomic modules and local systems \cite{Kas80, Kas84}. They were also vital in the proof of the Kazhdan–Lusztig conjecture \cite{BB81,BK81}. In this paper, we will show that the Gelfand-Kirillov dimension of $A\cV$-modules for the affine space $\A^n$ is greater or equal to $n$, a result analogous to the Bernstein inequality for $\cD$-modules. We define holonomic $A\cV$-modules as finitely generated modules that 
have Gelfand-Kirillov dimension $n$. 

In the case of the affine space $\A^n$, it was proved in \cite{XL23,BIN23} that $A\cV$-modules are modules over the tensor product $\cD \otimes U(\cL_+)$, where $\cD$ denotes the $n$-th Weyl algebra and $\cL_+$ denotes the Lie algebra of vector fields vanishing at the origin. In this paper, we will extensively use this isomorphism to study the holonomic $A\cV$-modules for the affine space.

The main result of this paper is Theorem \ref{theorem:simpleholonomicistensorproduct} that states that every simple holonomic $A\cV$-module is a tensor product of a simple holonomic $\cD$-module and a simple finite-dimensional $\gln$-module. We will also show in Theorem \ref{theorem:holonomicavmodhasfinitelength} that holonomic $A\cV$-modules have finite length. Finally, Theorem \ref{theorem:holonomicmodulesaredifferentiable} states that holonomic $A\cV$-modules are differentiable. This implies that holonomic modules can be associated with a bundle on the affine space that has a compatible action of the tangent bundle.

In Section \ref{section:preliminaries}, we introduce notations and define $A\cV$-modules. Preliminary results concerning the Gelfand-Kirillov dimension of $A\cV$-modules are presented in Section \ref{section:gkdimofavmod}. The main results of this paper are established in Section \ref{section:holonomicavmod}, where we focus on holonomic $A\cV$-modules. Following this, we discuss the notion of differentiable modules in Section \ref{section:differentiablemodules} and show that holonomic $A\cV$-modules are differentiable. Finally, in Section \ref{section:examplesofholonomicavmod}, we present examples of holonomic $A\cV$-modules.

\section*{Acknowledgments}	
The authors thank Kathlyn Dykes for numerous helpful discussions.
Y.B. gratefully acknowledges support with a Discovery grant from the Natural Sciences and Engineering Research Council of Canada.

\section{Preliminaries}
\label{section:preliminaries}

In this paper, $\bk$ is an algebraically closed field of characteristic $0$. Let $A = \bk[x_1,\dots,x_n]$ be the polynomial algebra in $n$ variables and let
\[
\cV = \Der(A) = \bigoplus_{i=1}^n \bk[x_1,\dots,x_n] \xpartial{x_i}
\]
be the Lie algebra of derivations of $A$. The associative algebra $U(\cV)$ is a Hopf algebra with coproduct $\Delta$. Furthermore, the commutative algebra $A$ is naturally a $U(\cV)$-module. The smash product $A \# U(\cV)$ is an associative algebra defined on the vector space $A \otimes U(\cV)$ with the product
\[
(f \otimes u)(g \otimes v)= \sum_{i} f u_i^{(1)}(g)  \otimes u_i^{(2)}v,
\]
where $\displaystyle \Delta(u) = \sum_{i} u_i^{(1)} \otimes u_i^{(2)}$. We will denote the element $f \otimes u \in A \# U(\cV)$ by $f \# u$ for $f \in A$ and $u \in U(\cV)$. With this notation,
\[
(f \# \eta)(g \# \mu) = f\eta(g) \# \mu + fg \# \eta \mu \ \ \ \text{for} \ f,g \in A, \ \eta,\mu \in \cV.
\]
For more details about Hopf algebras and smash products, we refer to \cite{DNR01}.

By \cite{XL23,BIN23}, the associative algebra $A \# U(\cV)$ is isomorphic to the tensor product $\cD \otimes U(\cL_+)$. The first algebra in this tensor product is the Weyl algebra $\cD$ of rank $n$. The algebra $A$ acts on itself by multiplication hence it can be seen as a subalgebra of $\End_{\bk}(A)$. Furthermore, $\cV$ is by definition a subspace of $\End_{\bk}(A)$. The algebra $\cD$ is the associative subalgebra of $\End_{\bk}(A)$ generated by the subspaces $A$ and $\cV$. The second tensor factor in the isomorphism is the universal enveloping algebra of the infinite-dimensional Lie algebra that we will construct as follows. Let $\cL$ be the Lie algebra of derivations of $\bk[X_1,\dots,X_n]$. This Lie algebra is isomorphic to $\cV$ but we would like to keep the distinction between $\cV$ and $\cL$. Consider the Lie subalgebra $\cL_+ =\fm \cL \subset \cL$, where $\fm$ is the maximal ideal in $\bk[X_1, \dots, X_n]$ generated by $X_1,\dots, X_n$. We have that $\cL_+$ is a graded Lie algebra $\displaystyle \cL_+ = \bigoplus_{d \geq 0} \cL_d$ and a basis of $\cL_d$ is given by 
\[
X^k \xpartial{X_i}, \ i \in \{1,\dots,n\}, k \in \Z_+, \ |k| = d +1.
\]

\begin{theorem}[{\cite{XL23,BIN23}}]\label{theorem:avisisotodtensorlplus}
The algebras $A \# U(\cV)$ and $\cD \otimes U(\cL_+)$ are isomorphic. The isomorphism maps $\varphi: A \# U(\cV) \rightarrow D \otimes U(\cL_+) $ and its inverse $\psi: D \otimes U(\cL_+) \rightarrow A \# U(\cV) $ defined by
\[
\varphi\left (x^k \xpartial{x_p} \right ) =  x^k\xpartial{x_p} \otimes 1 + \left (\sum_{0 < m \leq k } \binom{k}{m} x^{k-m} \otimes X^{m} \xpartial{X_p}\right ),
\]
$\varphi$ restricted to $A$ is the natural embedding into $\cD$, and 
\[
\psi|_{\cD} \left  (x^r \partial^s \right ) =  x^r \# \left (\xpartial{x_1} \right )^{s_1} \cdots \left (\xpartial{x_n} \right )^{s_n},
\]
\[
\psi|_{\cL_+} \left ( X^m \xpartial{X_p}\right ) = \sum_{0 \leq k \leq m} (-1)^{k} \binom{m}{k} x^{k} \# x^{m-k} \xpartial{x_p}.
\]
\end{theorem}

An $A\cV$-module is a module over the algebra $A \# U(\cV)$. Equivalently, an $A\cV$-module $M$ is a module over both $A$ and $\cV$ satisfying the Leibniz rule
\[
\eta \cdot (f \cdot m) = \eta(f) \cdot m + f \cdot (\eta \cdot m)
\]
for all $\eta \in \cV$, $f \in A$ and $m \in M$. We will heavily rely on Theorem \ref{theorem:avisisotodtensorlplus} to study $A\cV $-modules.

\begin{example}
Let $P$ be a $\cD$-module and $Q$ an $\cL_+$-module, then, by Theorem \ref{theorem:avisisotodtensorlplus}, $P \otimes Q$ is an $A\cV$-module with the action
\begin{align*}
x^k (p \otimes q ) &= x^k p \otimes q, \\
x^k \xpartial{x_i} (p \otimes q ) &= \left (x^k \xpartial{x_i} p \right ) \otimes q 
+ \sum_{0 < m \leq k } \binom{k}{m} \left ( x^{k-m} p \right )\otimes X^{m} \xpartial{X_i} q,
\end{align*}
for  $p\in P, q \in Q$. 
\end{example}

\section{Gelfand-Kirillov dimension of $A\cV$-modules}\label{section:gkdimofavmod}

Let $R$ be a finitely generated unital associative algebra and $M$ be a finitely generated left $R$-module. Suppose that $C \subset R$ is a finite-dimensional generating subspace of $R$ containing $1$ and $M_0$ is a finite-dimensional generating subspace of $M$. The subspace $C^m$ spanned by the products of elements of $C$ of length $m$, is a finite-dimensional vector space and defines an increasing filtration on $R$
\[
C \subset C^2 \subset C^3 \subset 
\]
such that $\displaystyle R = \bigcup_{m \geq 1} C^m$. Similarly, $M_k = C^kM_0 = \left \{rv  \mid r \in C^k, \ v \in M  \right \}$ defines an increasing filtration on $M$ such that $C^mM_k \subset M_{k+m}$ for every $m,k \geq 1$. The Gelfand-Kirillov dimension $\GKdim_{R}(M)$ of $M$ is the superior limit 
\[
\GKdim_R (M) = \mathop{\overline{\lim}}\limits_{m  \rightarrow \infty} \log_m(\dim C^m M_0 ) = \lim_{k \rightarrow \infty} \sup_{m \geq k} \log_m(\dim C^m M_0 ).
\]
It is well-known that the Gelfand-Kirillov dimension does not depend on the choice of $C$ or $M_0$.

\begin{example}\label{example:bernsteininequalityfordmodules}
Let $M$ be a non-zero finitely generated $\cD$-module, then it is well-known that $ n \leq \GKdim_{\cD}(M) \leq 2n$, see \cite[Theorem 9.4.2]{Cou95} for instance. This inequality is called the Bernstein inequality.
\end{example}
\begin{lemma}\label{lemma:gkdimlessthanoneimpliesfd}
Let $R$ be a finitely generated associative algebra and $M$ be a finitely generated module. If $\GKdim_R(M) < 1$, then $\dim M < \infty$. 
\end{lemma}
\begin{proof}
Let $C$ be a generating finite-dimensional subspace of $R$ containing $1$ and $M_0$ a finite-dimensional generating subspace of $M$ as a $R$-module. We have that $\dim C^mM_0 \leq \dim C^{m+1} M_0$ for every $m \geq 0$. If $\dim C^m M_0 < \dim C^{m+1}M_0$ for evert $m \geq 0$, then $\dim C^mM_0 \geq m$. Hence, $\GKdim_R(M) \geq \log_m(m) =1$. Therefore, there exists $m$ such that $C^m M_0= C^{m+1}$. Because $C^mM_0$ is an invariant subspace of $M$ that contains a generating set of $M$, $M = C^m M_0$. We conclude that $\dim M = \dim C^m M_0 < \infty$.
\end{proof}

We can use this lemma to give a description of $\cL_+$-modules with zero Gelfand-Kirillov dimension. 

\begin{corollary}\label{lemma:gkdimlessthanoneimpliesfdforlplus}
Let $W$ be a finitely generated $\cL_+$-module with $\GKdim(W) =0$, then $W$ is a finite-dimensional $\cL_+$-module. 
\end{corollary}

Let $B$ the subspace of $\cD$ generated by $1$, $x_1,\dots,x_n$ and $\xpartial{x_1}, \dots, \xpartial{x_n}$, then $B$ defines a filtration on $\cD$ called the \emph{Bernstein filtration}. Define 
\[
F'= \begin{cases} \cL_0 \oplus \cL_1\oplus \cL_2 & \text{if} \ n=1,\\
\cL_0 \oplus \cL_1 & \text{if} \ n>1.
\end{cases}
\]
The proof of the following lemma is straightforward and is left for the reader.
\begin{lemma}\label{lemma:generatorsoflplus}
$F'$ generates $\cL_+$ as a Lie algebra. In particular, if $n>1$, then the subalgebra $\displaystyle \cL_{++} = \bigoplus_{k \geq 2} \cL_k$ is generated by $\cL_{1}$.
\end{lemma}

Because $F'$ generates $\cL_{+}$ as a Lie algebra, $F =F' \cup \{1 \}$ generates the associative algebra $U(\cL_+)$. Therefore, both $\cD$ and $U(\cL_+) $ are finitely generated algebras, hence $D \otimes U(\cL_+)$ is a finitely generated algebra as well.  Let $G = B \otimes 1 + 1 \otimes F$. Then $G$ is a finite-dimensional generating subspace of $D \otimes U(\cL_+)$.

\begin{lemma}\label{lemma:bernsteininequalityforav}
If $M$ is a non-zero finitely generated $A\cV$-module, then $\GKdim_{A\cV} (M) \geq n$.
\end{lemma}
\begin{proof}
 Suppose that $M_0$ is a finite-dimensional generating subspace of $M$ as a module over $A\# U(\cV) \cong D \otimes U(\cL_+)$. Then,
\[
B^m M_0 \subset G^m M_0.
\]
Therefore, $\GKdim_{\cD} (\cD M_0) \leq \GKdim_{A\cV}(M)$. 
Since $\GKdim_{\cD}(M_0) \geq n$, we obtain the claim of the Lemma.
\end{proof}

\begin{lemma}\label{lemma:gkdimoftensorproduct}
 Let $P$ be a finitely generated $\cD$-module and $Q$ be a finitely generated $\cL_+$-module. Then, \[\GKdim_{A\cV}( P \otimes Q) = \GKdim_{\cD}(P) + \GKdim_{U(\cL_+)}(Q). \]
\end{lemma}
\begin{proof}
Let $P_0 \subset P$ and $Q_0 \subset Q$ be finite-dimensional generating subspace of $P$ and $Q$ as a $\cD$-module and $U(\cL_+)$-module, respectively. We have that $P_0 \otimes Q_0$ is a finite-dimensional generating subspace of $P\otimes Q$ as $\cD \otimes U(\cL_+)$-module and
\[
B^m P_0 \otimes F^m G \subset G^{2m} (P_0 \otimes Q_0) \subset B^{2m} P_0 \otimes F^{2m}Q_0.
\]
Hence,
\[
\dim(B^m P_0) \dim(F^m Q_0 ) \leq \dim (G^{2m} (P_0 \otimes Q_0)) \leq \dim (B^{2m} P_0)\dim (F^{2m}Q_0).
\]
 By \cite[Chapter 9]{Cou95}, the dimension of $B^m P$ for large values of $m$ is given by a polynomial, thus the Gelfand-Kirillov dimension is the degree of this polynomial. Hence,
 \[
 \GKdim_{\cD} (P) = \lim_{m \rightarrow \infty} \log_m B^mP_0.
 \]
Furthermore, if $(x_k)$ and $(y_k)$ are sequences of real numbers such that $(y_k)$ converges, then 
 \[
 \mathop{\overline{\lim}}\limits_{k  \rightarrow \infty} (x_k+y_k)  = \mathop{\overline{\lim}}\limits_{k  \rightarrow \infty} \, x_k + \lim_{k \rightarrow \infty} y_k 
 \]
 Finally, because 
 \[
 \mathop{\overline{\lim}}\limits_{m  \rightarrow \infty} \log_{m} x_m =\mathop{\overline{\lim}}\limits_{m  \rightarrow \infty} \frac{\log x_{2m}}{\log 2m} =  \mathop{\overline{\lim}}\limits_{m  \rightarrow \infty} \frac{\log x_{2m}}{\log m + \log 2} = \mathop{\overline{\lim}}\limits_{m  \rightarrow \infty}\frac{\log x_{2m}}{\log m} = \mathop{\overline{\lim}}\limits_{m  \rightarrow \infty} \log_m(x_{2m}),
 \]
 we get that
 \begin{align*}
    \GKdim_{A\cV}(P \otimes Q)    = &\mathop{\overline{\lim}}\limits_{m  \rightarrow \infty} \log_{m} \dim(G^{2m} (P_0 \otimes Q_0)) \\
    = & \mathop{\overline{\lim}}\limits_{m  \rightarrow \infty} \log_m \left (  \dim(B^mP_0) \dim(F^m Q_0 ) \right ) \\
    = & \mathop{\overline{\lim}}\limits_{m  \rightarrow \infty} \left ( \log_m  \dim(B^mP_0)  + \log_m \dim(F^m Q_0  ) \right) \\
    = & \lim_{m \rightarrow \infty } \log_m  \dim(B^m P_0) + \mathop{\overline{\lim}}\limits_{m  \rightarrow \infty} \log_m \dim(F^m Q_0 ) \\
    = & \GKdim_{\cD}(P) + \GKdim_{U(\cL_+)}(Q).
 \end{align*}
\end{proof}

\section{Holonomic $A\cV$-modules}\label{section:holonomicavmod}

In Section \ref{section:gkdimofavmod}, we proved that $\GKdim_{A\cV}(M) \geq n$ if $M$ is a finitely generated $A\cV$-module. We say that a finitely generated $A\cV$-module $M$ is \emph{holonomic} if $M =0$ or $\GKdim_{A\cV}(M) = n$.

\begin{lemma}\label{lemma:exactgkdim}
If $M$ is a holonomic $A\cV$-module and $N \subset M$ is a submodule, then $N$ and $M/N$ are holonomic. Furthermore, finite sums of holonomic $A\cV$-modules are holonomic.
\end{lemma}
\begin{proof}
    Let $M$ be a holonomic $A\cV$-module and $N \subset M$ a submodule of $M$. If $N$ is non-zero, $\GKdim_{A\cV} (N) \geq n$ by Lemma \ref{lemma:bernsteininequalityforav}. Since $N$ is a submodule of $M$,  $\GKdim_{A\cV} (N) \leq \GKdim_{A\cV}(M) = n$. Therefore, $\GKdim_{A\cV} (N) = n$. By the properties of the Gelfand-Kirillov dimension (see \cite[Proposition 5.1]{KL00}), $\GKdim_{A\cV} (M/N) \leq \GKdim_{A\cV} (M)=n$. By Lemma \ref{lemma:bernsteininequalityforav}, $\GKdim_{A\cV} (M/N)$ is either zero or $n$. In either case, $M/N$ is holonomic. 
    
    For the second claim of the Lemma, by \cite[Proposition 5.1]{KL00} and Lemma \ref{lemma:bernsteininequalityforav}, if $M_1,\dots,M_r$ are non-zero holonomic $A\cV$-modules and $M'= \sum_{i=1}^r M_i$, then 
    \[
    n \leq \GKdim_{A\cV}(M') \leq \max_{i}\GKdim_{A\cV}(M_i) = n.
    \]
    Therefore, $ \GKdim_{A\cV}(M') =n$ and $M'$ is holonomic.
\end{proof}

\begin{proposition}\label{proposition:fingenamodarehol}
Let $M$ be an $A\cV$-module that is finitely generated as an $A$-module. Then, $M$ is a holonomic $A\cV$-module. 
\end{proposition}
\begin{proof}
Let $C$ be a generating subspace for the algebra $A \# U(\cV)$, containing $\{ x_1, \ldots, x_n\}$. By \cite{BIN23}, $M$ is a free module over $A$ of a finite rank, $M = A \otimes W$ with $\dim W < \infty$. Then $M$ is generated by the finite-dimensional subspace $M_0 = 1 \otimes W$. Consider filtration $M_0 \subset M_1 \subset M_2 \subset \ldots$,
where $M_{r+1} = C M_r$. Clearly, $M_r$ contains all elements of the form 
$x^k \otimes w$ with $|k| \leq r$. The action of $\cV$ on $M$ was explicitly described in \cite{BIN23} (see also Section \ref{section:examplesofholonomicavmod} below). From this description we see that there is a linear in $r$ bound on the degree in $x$ of the elements in $M_r$. This implies that $\lim\limits_{r \to \infty} \log_r \dim M_r = n$.
\end{proof}

\begin{lemma}\label{lemma:tensorholiffishol}
Let $M$ be an $A\cV$-module such that $M \cong P \otimes Q$, where $P$ is a $\cD$-module and $Q$ is an $\cL_+$-module. Then, $M$ is a holonomic $A\cV$-module if and only if $P$ is a holonomic $\cD$-module and $Q$ is a finite-dimensional $\cL_+$-module.
\end{lemma} 
\begin{proof}
Assume that $M$ is a holonomic $A\cV$-module. By Lemma \ref{lemma:gkdimoftensorproduct},
\[
n = \GKdim_{A\cV}(M) = \GKdim_{\cD} (P) + \GKdim_{U(\cL_+)} (Q) 
\]By the Bernstein inequality for $\cD$-modules, $\GKdim_{\cD}(P) \geq n$. Hence, $ \GKdim_{\cD} (P)=n $ and $P$ is a holonomic $\cD$-module. Furthermore, $\GKdim_{U(\cL_+)} (Q) = 0$, thus $Q$ is a finite-dimensional vector space by Corollary \ref{lemma:gkdimlessthanoneimpliesfdforlplus}. 

On the other hand, if $P$ is a holonomic $\cD$-module and $Q$ is a finite-dimensional $\cL_+$-module, then $M$ is a holonomic $A\cV$-module by Lemma \ref{lemma:gkdimoftensorproduct} and Corollary \ref{lemma:gkdimlessthanoneimpliesfdforlplus}.
\end{proof}

\begin{lemma}
Holonomic $A\cV$-modules are cyclic.
\end{lemma}
\begin{proof}
Suppose that $M_0$ is a finite-dimensional generating subspace of a holonomic $A\cV$-module $M$. Since $\GKdim_{\cD} (\cD M_0) \leq \GKdim_{A\cV}(M) =n$,
we have that $\cD M_0$ is a holonomic $\cD$-module. Therefore, there exists $v \in \cD M_0$ such that $\cD M_0 = \cD v$ because every holonomic $\cD$-module is cyclic \cite[Corollary 10.2.6]{Cou95}. We conclude that $M$ is generated by $v$ as a module over $\cD \otimes U(\cL_+)$ because $M = U(\cL_+)\cD M_0 = U(\cL_+)\cD v$.
\end{proof}

\begin{corollary}
Holonomic $A\cV$-modules are Noetherian.
\end{corollary}
\begin{proof}
Since every submodule of a holonomic $A\cV$-module is holonomic, we have that every submodule of a holonomic $A\cV$-module is cyclic by last lemma. Hence, a holonomic $A\cV$-module is Noetherian because every submodule is finitely generated (see \cite[Section VII.1.1 Proposition 2]{Bou22}).
\end{proof}

\begin{lemma}[Schur's Lemma]\label{lemma:schurslemmafordmodules}
Suppose $\bk$ is uncountable. Let $P$ be a simple $\cD$-module, then $\End_{\cD}(P) \cong \bk$.
\end{lemma}
\begin{proof}
If $P$ is a simple $\cD$-module, then $P$ has countable dimension as a $\bk$-vector space. Assuming that $\bk$ is uncountable, it follows from standard applications of the Schur's Lemma that $\End_{\cD}(P) \cong \bk$ (see \cite[Section VIII.3.2 Theorem 1]{Bou22}). 
\end{proof}
From now on, we assume that $\bk$ is uncountable.

\begin{lemma}[Jacobson Density Theorem]
Let $P$ be a simple $\cD$-module, $\{p_1,\dots,p_k \} \subset P$ a linearly independent set and $w_1,\dots,w_k \in P$. Then, there exists $f \in \cD$ such that $f p_i = w_i$ for all $i = 1, \ldots, k$.
\end{lemma}
\begin{proof}
This is an application of the Jacobson Density Theorem (see \cite[Theorem 13.14]{Isa09}) using the Schur's Lemma \ref{lemma:schurslemmafordmodules}.
\end{proof}

\begin{proposition}\label{proposition:holonomicmoduleshavetensorsub}
Let $M$ be a holonomic $A\cV$-module, then there exists an $A\cV$-submodule $N \subset M$ such that $N \cong P \otimes Q$, where $P$ is a simple holonomic $\cD$-module and $Q$ is a finite-dimensional $\cL_+$-module.
\end{proposition}
\begin{proof}
Suppose that $M_0$ is a finite-dimensional generating subspace of $M$. Then, $\cD M_0$ is a holonomic $\cD$-module. Therefore, $\cD M_0$ is an Artinian $\cD$-module by \cite[Theorem 10.2.2]{Cou95}, thus contains a simple $\cD$-module $P$, which is also a holonomic $\cD$-module. The map
\begin{align*}
\Psi: P \otimes U(\cL_+) & \rightarrow M \\
p \otimes u & \mapsto up
\end{align*}
is a homomorphism of $A\cV$-modules, where we view $U(\cL_+)$ as a left $U(\cL_+)$-module. Let $w \in \mathrm{ker}\, \Psi$. Write
\[
w =  \sum_{i=1}^k p_i \otimes u_i
\]
with $\{p_1,\dots,p_k\}$ linearly independent and $u_1,\dots,u_k \in U(\cL_+)$ non-zero. By Jacobson density theorem, there exists $f \in D$ such that $f p_i = \delta_{1i} p_1$. Hence, $fw = p_1 \otimes u_1$. We conclude that $P \otimes u_1 \subset Dw$. Likewise, $P \otimes u_i \subset Dw$ for all $i = 1, \ldots, k$. Therefore $(\cD \otimes U(\cL_+))w = P \otimes I_w$, where $I_w$ is a left ideal of $U(\cL_+)$. Define $\displaystyle I = \sum_{w \in \mathrm{ker} \Psi} I_w$. Hence, the image of $\Psi$ is an $A\cV$-submodule of $M$ isomorphic to 
\[
(P \otimes U(\cL_+))/ (P \otimes I) \cong P \otimes Q.
\]
where $Q =  U(\cL_+)/I$. By Lemma \ref{lemma:gkdimoftensorproduct}, $\GKdim_{U(\cL_+)}(Q)=0$ because $\GKdim_{A\cV}(M) = \GKdim_{\cD}(P) =n$. Hence, $Q$ is a finite-dimensional generalized weight $\cL_+$-module by Corollary \ref{lemma:gkdimlessthanoneimpliesfdforlplus}.

\end{proof}

\begin{theorem}\label{theorem:simpleholonomicistensorproduct}
Let $M$ be a simple holonomic $A\cV$-module, then $M \cong P \otimes Q$, where $P$ is a simple holonomic $\cD$-module and $Q$ is a simple finite-dimensional $\gl_n$-module.
\end{theorem}
\begin{proof}
By Proposition \ref{proposition:holonomicmoduleshavetensorsub}, there exists an isomorphism $\psi : P \otimes Q \rightarrow M$, where $P$ is a simple holonomic $\cD$-module and $Q$ is a finite-dimensional $\cL_+$-module. If $Q' \subset Q$ is an $\cL_+$-submodule of $Q$, then $\psi(P \otimes Q')$ is a proper $A\cV$-submodule of $M$. Therefore, $Q$ must be a simple $\gl_n$-module because every simple finite-dimensional $\cL_+$-module is a simple $\gl_n$-module.
\end{proof}

\begin{lemma}\label{lemma:tensorhasfinitelength}
Let $M$ be a holonomic $A\cV$-module that is isomorphic to the tensor product $P \otimes Q$, where $P$ is a $\cD$-module and $Q$ is an $\cL_+$-module. Then, $M$ has finite length.
\end{lemma}
\begin{proof}
If $M \cong P \otimes Q$, then $P$ is a holonomic $\cD$-module and $Q$ is a finite-dimensional $\cL_+$-module by Lemma \ref{lemma:tensorholiffishol}. Since $P$ is a holonomic $\cD$-module, $P$ has finite length by \cite[Section 10.2]{Cou95}. Let $P_1 \subset P_2 \subset \dots \subset P_r = P$ a composition series of $P$. On the other hand, because $Q$ is a finite-dimensional $\cL_+$-module, it admits a composition series  $Q_1 \subset Q_2 \subset \dots \subset Q_s =Q$ as well. If $N_{a,1} = P_a \otimes Q_1 $ and $N_{a,b} = P_a \otimes Q_b + P_r \otimes Q_{b-1}$ for each $a=1,\dots,r$ and $b= 2,\dots,s$, then 
\[
N_{1,1} \subset N_{2,1} \subset \dots \subset N_{r,1} \subset N_{1,2} \subset N_{2,2} \subset \dots \subset N_{r,2} \subset N_{1,3} \subset \dots \subset N_{r-1,s} \subset N_{r,s} = M
\]
is a composition series for $P \otimes Q$.
\end{proof}

\begin{theorem}\label{theorem:holonomicavmodhasfinitelength}
Let $M$ be a holonomic $A\cV$-module. Then, $M$ has finite length.
\end{theorem}
\begin{proof}
By Proposition \ref{proposition:holonomicmoduleshavetensorsub},
$M$ has a non-zero submodule $N_1$ of the form $N_1 = P_1 \otimes Q_1$,
where $P_1$ is a $\cD$-module and $Q_1$ is an $\cL_+$-module. The quotient $M/N_1$ is holonomic and has a non-zero submodule $\overline{N}_2$, which is a tensor product $\overline{N}_2 = P_2 \otimes Q_2$. Let $N_2$ be the preimage of $\overline{N}_2$ under
$M \rightarrow M/N_1$. Continuing this inductive procedure, we will get a strictly increasing chain of submodules
\[
N_1 \subset N_2 \subset N_3 \subset \ldots
\]
with holonomic quotients $N_i / N_{i-1} \cong P_i \otimes Q_i$, where 
$P_i$ are holonomic $\cD$-modules and $Q_i$ are finite-dimensional $\cL_+$-modules. Since $M$ is Noetherian, this sequence terminates, $N_k = M$.

By Lemma \ref{lemma:tensorhasfinitelength}, each $N_{i}/N_{i-1}$, $i=1,\dots,k$ has finite length. Therefore, $M$ has a finite length.
\end{proof}

Suppose that $M$ is an $A\cV$-module. By Theorem \ref{theorem:avisisotodtensorlplus}, $M$ is a module for $A\# U(\cV)$, which is isomorphic to $\cD \otimes U(\cL)$. In particular, $M$ is a module over $\cD \cong \cD \otimes \bk \subset \cD \otimes U(\cL)$. Combining Lemma \ref{lemma:tensorholiffishol}, Theorem \ref{theorem:simpleholonomicistensorproduct} and Theorem \ref{theorem:holonomicavmodhasfinitelength}, we get that every holonomic $A\cV$-module is  holonomic as a $\cD$-module as well, because Lemma \ref{lemma:exactgkdim} and Theorem \ref{theorem:holonomicavmodhasfinitelength} also hold for holonomic $\cD$-modules \cite[Proposition 10.1.1, Theorem 10.2.2]{Cou95}.
\begin{corollary}
Let $M$ be a holonomic $A\cV$-module. Then, $M$ is holonomic as a $\cD$-module.
\end{corollary}

\section{Holonomic modules are differentiable}\label{section:differentiablemodules}
Let $R$ be a commutative algebra over $\bk$ and let $M,N$ be $R$-modules.
For $f \in R$ denote by $\delta(f)$ the adjoint operator on $\Hom_{\bk}(M,N)$,
$\delta(f) D = D\circ f - f \circ D$.
Following \cite{Gro67}, we set $\Diff_{-1}(M,N) = (0)$ and define inductively
\[
\Diff_{s+1}(M,N) = \left \{D \in \Hom_{\bk}(M,N) \mid \delta(f) D \in \Diff_{s} (M,N) \ \ \forall f \in R \right \}.
\]
for each $s \geq -1$. 
Note that $\Diff_{0}(M,N) = \Hom_R (M, N)$.
An element of $\Diff_s(M,N)$ is a called a \emph{differential operator} of order less or equal to $s$ and 
we define the space of differential operators as
\[
\Diff ( M,N) = \bigcup_{s \geq 0} \Diff_k(M,N).
\] 
Then differential operators of order less or equal to $s$ are determined by the condition
\[
\delta(f_1) \circ \cdots \circ \delta(f_s) (D) = 0 \ \ \ \ \text{for all} \ f_1, \ldots, f_s \in R.
\]
Since the field $\bk$ is infinite, a linearization argument shows that this is equivalent to 
\begin{equation}
\label{sDiff}
\delta(f)^s D = 0 \ \ \ \ \text{for all} \  f \in R.
\end{equation}

\begin{lemma}
\label{gen}
Let $\{x_1, \ldots, x_n \}$ be a set of generators of $R$. Then $D \in \Hom_{\bk}(M,N)$ is a differential operator of order $s$ if and only if 
\[
\delta(x)^m D = 0
\]
for all $m \in \Z_+^n$ with $|m| = s$.
\end{lemma}
Here $\delta(x)^m = \delta(x_1)^{m_1} \circ \ldots \circ \delta(x_n)^{m_n}$.
\begin{proof}
This follows from the relation
\[
\delta(fg) D = f \circ \delta(g) D + (\delta(f) D) \circ g.  
\]
\end{proof}

If $M$ is an $A\cV$-module, then $M$ comes with a representation $\rho:\cV \rightarrow \gl(M)$ of the Lie algebra $\cV$. The map $\rho$ is a $\bk$-linear map between the $A$-modules $\cV$ and $\gl(M)$. If $\rho$ is $A$-linear, then $M$ is a $\cD$-module. 
In general, $\rho$ need not be a homomorphism of $A$-modules, but it could still be a differential operator. 

We say that an $A\cV$-module $M$ is \emph{differentiable} if the representation $\rho:\cV \rightarrow \gl(M)$ associated to it is a differential operator. 
By (\ref{sDiff}) an $A\cV$-module $M$ is $s$-differentiable if and only if it is annihilated by 
\[
\sum_{j=0}^s (-1)^{j} \binom{s}{j} f^{j} \circ \rho(f^{s-j} \eta)
\]
for all $f \in A$, $\eta \in \cV$.

It was proved in \cite{BR23} in the setting of an arbitrary smooth affine variety that if an $A\cV$-module $M$ is finitely generated as $A$-module, then $M$ is differentiable. 
Later in this section we will prove that holonomic $A\cV$-modules are differentiable.

\begin{lemma}
\label{mL}
An $A\cV$-module $M$ is $s$-differentiable if an only if $\fm^s \cL_+$ annihilates $M$.
\end{lemma}
\begin{proof}
It follows from Theorem \ref{theorem:avisisotodtensorlplus} that the action of $X^m  \xpartial{X_i}$ on $M$ is given by $(\delta(x)^m \rho) \xpartial{x_i}$.
Now the claim follows from Lemma \ref{gen}.
\end{proof}

\begin{lemma}
\label{Lp-ann}
Let $S$ be a submodule in an $\cL_+$-module $Q$. If $S$ is annihilated by 
$\fm^s \cL_+$ and $Q/S$ is annihilated by $\fm^r \cL_+$ then $Q$ is annihilated
by $\fm^q \cL_+$, where $q = 2\max(s,r)$ if $n > 1$ 
and $q = 2\max(s,r) + 1$ if $n=1$.
\end{lemma}
\begin{proof}
It is straightforward to check that 
\[
\left[ \fm^s \cL_+, \fm^s \cL_+ \right] = \fm^{2s} \cL_+,
\]
unless $n=1$, in which case the right hand side is $\fm^{2s+1} \cL_+$. This implies the claim of the Lemma.
\end{proof}

\begin{theorem}\label{theorem:holonomicmodulesaredifferentiable}
Any holonomic $A\cV$-module $M$ is differentiable.
\end{theorem}
\begin{proof}
By Theorem \ref{theorem:holonomicavmodhasfinitelength}, $M$ has finite length. Let \[
0 = N_0 \subset N_1 \subset \dots \subset N_r = M
\]
be a composition series for $M$. Then by Theorem \ref{theorem:simpleholonomicistensorproduct}, $N_i / N_{i-1} \cong P_i \otimes Q_i$ as a module over $\cD \otimes U(\cL_+)$, where $P_i$ is a simple holonomic $\cD$-module and $Q_i$ is a finite-dimensional simple $U(\cL_+)$-module. Simple finite-dimensional $U(\cL_+)$-modules are $\gln$-modules, and are annihilated by $\fm \cL_+$. 
Thus each quotient $N_i / N_{i-1}$ is annihilated by $\fm \cL_+$.
Applying repeatedly Lemma \ref{Lp-ann}, we conclude that $M$ is annihilated 
by $\fm^s \cL_+$ for some $s$.
By Lemma \ref{mL}, $M$ is differentiable.

\end{proof}

Differentiable $A\cV$-modules admit a sheafification into a quasicoherent sheaf. Let us outline this procedure.
Let $f\in A$ nonzero, then the localization
\[
A_f = \left \{ \frac{g}{f^k}\mid g\in A, \ k \geq 0  \right \}
\]
is an algebra. If $M$ is an $A$-module, then $M_f = A_f \otimes_A M$ is an $A_f$-module. Suppose that $M$ is an $A\cV$-module, then it is not necessarily true that $M_f$ is an $A_f\cV_f$-module, i.e. a module over the associative algebra $A_f \# U(\cV_f)$. However, if $M$ is finitely generated as $A$-module, this holds true by \cite{BR23}. We can use Theorem \ref{theorem:holonomicmodulesaredifferentiable} to show that same holds for differentiable $A\cV$-modules.
\begin{corollary}[cf. \cite{BI23, BR23}]
Let $M$ be an $s$-differentiable $A\cV$-module with associated representation $\rho : \cV \rightarrow \gl(M)$ and let $f\in A$ with $f \neq0$. Then, $M_f$ is a module over $A_f \# U(\cV_f)$. Explicitly, the action of $A_f$ on $M_f$ is given by localization and 
\[
\left( \frac{g}{f^k} \xpartial{x_i}\right) m = \sum_{p=0}^{s} \sum_{l=0}^p (-1)^l \binom{p+k}{p} \binom{p}{k} \frac{1}{f^{k+l}} \rho\left(f^lg \xpartial{x_i}\right)m
\]
for every $g \in A$, $k \geq 0$ and $m\in M_f$.
\end{corollary}

\section{Examples of holonomic $A\cV$-modules}
\label{section:examplesofholonomicavmod}

Let $M$ be an $A\cV$-module that is finitely generated as an $A$-module. By \cite[Theorem 5.4]{BIN23}, $M$ is a gauge module. A gauge module is an $A\cV$-module that is a free $A$-module of finite rank, $M \cong A \otimes W$, that admits an $A$-linear representation $\rho: \cL_+ \rightarrow \gl_{A}(M) $ and $A$-linear maps $B_i: A \otimes W \rightarrow A \otimes W$, $i=1,\dots,n$ with
\[
\left [ \xpartial{x_i} \otimes 1 + B_i , \xpartial{x_j} \otimes 1 + B_j\right ] =0 \ \text{and}\
\left [B_i , \rho(\cL_+) \right ] =0
\]
such that $f(g \otimes w) = (fg) \otimes w$ and
\begin{equation*}
\left (x^m \xpartial{x_i} \right )(g \otimes w ) = x^m \fxpartial{g}{x_i} w + x^m g B_i(w) + g\sum_{0 < k \leq m} \binom{m}{k} x^{m-k} \rho\left ( X^k \xpartial{X_i} \right )w.
\end{equation*}
By Proposition \ref{proposition:fingenamodarehol}, gauge modules are holonomic $A\cV$-modules because they are finitely generated as $A$-modules.

Another family of $A\cV$-modules that was studied in the literature is the \emph{Rudakov modules}, see \cite{BFN19,BB24}. Let $p \in \mathbb{A}^n$ and $P$ be the $\cD$-module of delta-functions supported at $p$. The module $P$ is spanned by all partial derivatives of delta-functions $\delta_p$, i.e.
\[
P \cong \bk\left [\xpartial{x_1},\dots,\xpartial{x_n} \right ]\delta_p.
\]
From this realization, we immediately see that $P$ is a holonomic $\cD$-module. Let $W$ be any finite-dimensional $\cL_+$-module. The Rudakov module $R_p(W)$ is the tensor product
\[
R_p(W) = P \otimes W\cong \bk\left [\xpartial{x_1},\dots,\xpartial{x_n} \right ]\delta_p \otimes W.
\]
This is an example of a finitely generated $A\cV$-module that is not finitely generated as an $A$-module. Because $P$ is a holonomic $\cD$-module and $W$ is finite-dimensional, we have that Rudakov modules are holonomic $A\cV$-modules.

By Theorem \ref{theorem:simpleholonomicistensorproduct}, a simple holonomic $A\cV$-module is the tensor product of a holonomic $\cD$-module and a simple finite-dimensional $\gl_n$-module. This construction played a crucial role in the classification of weight modules over $\cV$ with finite weight multiplicities established in \cite{GS22} as well in the classification of such modules over the Lie algebra of vector fields on a torus \cite{BF16}. In particular, $A\cV$-modules that appear in these papers are holonomic.



\end{document}